\documentclass{amsart}
\usepackage{indentfirst,latexsym,bm,amsmath,amssymb,amsthm}
\usepackage{enumerate}
\usepackage{color}
\newtheorem{theorem}{Theorem}[section]
\newtheorem{lemma}[theorem]{Lemma}

\theoremstyle{definition}

\theoremstyle{remark}
\newtheorem{remark}[theorem]{Remark}

\numberwithin{equation}{section}

\begin{document}

\title{Blow up of fractional Schr\"odinger equations on manifolds with nonnegative Ricci curvature}

\author{Huali Zhang}
\address{School of Mathematics and Statistics, Changsha University
of Science and Technology, Changsha 410114, Peoples Republic of China.}
\email{zhlmath@yahoo.com}

\author{Shiliang Zhao}
\address{School of Mathematical Sciences, Sichuan University, Chengdu 610064, Peoples Republic of China.}
\email{zhaoshiliang@scu.edu.cn}

\subjclass[2010]{Primary 35A01}

\date{\today}


\keywords{Fractional Schr\"odinger equations, blow-up, weight function, heat kernel.}

\begin{abstract}
In this paper, the well-posedness of Cauchy's problem of fractional Schr\"odinger equations with a power type nonlinearity on $n$-dimensional manifolds with nonnegative Ricci curvature is studied. Under suitable volume conditions, the local solution 
with initial data in $H^{[\frac{n}{2}]+1}$
will blow up in finite time no matter how small the initial data is, which follows from a new weight function and ODE inequalities. Moreover, the upper-bound of the lifespan can be estimated.
\end{abstract}

\maketitle




\section{\textbf{Introduction and Main results}}
Let $0<\alpha<2$ and in this short article we study the Cauchy's problem of nonlinear fractional Schr\"odinger equations
\begin{eqnarray}\label{NFSE}
\left
    \{
        \begin{array}{l}
                      \mathrm{i} \partial_t u- (-\Delta)^{\frac{\alpha}{2}} u=F(u,\overline{u}), \ (t,x) \in (0, + \infty) \times M \\
                       u|_{t=0}=\varphi_0(x)+\mathrm{i}\varphi_1(x),\\

        \end{array}
\right.
\end{eqnarray}
where $\varphi_0(x)$, $\varphi_1(x) \in C_c^\infty (M)$ are real valued functions and $M$ is a complete manifold with nonnegative Ricci curvature. $F$ is a nonlinear function of $u$ ,$\bar{u}$, satisfying
\begin{equation}\label{F}
  |F(u,\overline{u})|\lesssim |u|^p,
\end{equation}
where $\bar{u}$ is the conjugate of $u$ and $p>1.$ Note that when $\alpha=2$, \eqref{NFSE} is the nonlinear Schr\"odinger equation. When $1<\alpha<2$,
 \eqref{NFSE} was introduced by Laskin in \cite{Laskin}. Namely, the quantum mechanics path integral over Brownian trajectories leads to the
well known Schr\"odinger equation ($\alpha=2$), and the path integral over
L$\mathrm{\acute{e}}$vy trajectories leads to the fractional Schr\"odinger equation ($1<\alpha<2$).

First we recall some known results about the nonlinear Schr\"odinger equations on $\mathbb{R}^n$. The global existence results for small data can be established when $p$ is large enough. If $p>\frac{\sqrt{n^2+12n+4}+n+2}{2n}$, it was shown by  Strauss \cite{Strauss} that the global solutions exist for small initial data. For $F=|u|^{p-1}u$, the wave operators can be constructed in general for small data when $p>1+\frac{2}{n}$  \cite{GOV}, \cite{Nakanishi}. And hence the global existence is guaranteed by the conservation of mass and energy and the wave operators.  When $p$ is small, the structure of the nonlinearity starts to play a crucial role. For example, consider $F=|u|^{p-1}u$ and $1<p<1+\frac{2}{n}$. For $1\le n\le 3$, asymptotically free solutions can not exist \cite{Haya1,Ozawa}.  For $F=|u|^{p}$, Ikeda-Inui \cite{Ikeda} proved a small data blow-up result of $L^2$ when $1<p<1+\frac{4}{n}$. For other nonlinearities such as $u^2, \bar{u}^2$, few examples of global existence for small data below the Strauss exponent are known. We refer the readers to \cite{CW,GMS,GMS2,GH,HN,Kawahara,MS} and references therein. Furthermore, there are some Strichartz estimates for nonlinear Schr\"odinger equations on manifolds, for instances \cite{AP,BD,BHL,Banica,Banica1,Blair,Bruq1,Bruq2,Bruq3,Fontaine1,Pierfelice,Ionescu1,Ionescu2}.

As for nonlinear fractional Schr\"odinger equations, the Strichartz estimates was established by Cho-Koh-Seo \cite{Cho} in the radial case. For $F=|u|^{p-1}u$,
Guo-Sire-Wang-Zhao \cite{GSWZ} showed the global
well-posedness of radial solutions in the energy critical case. Boulenger-Himmelsbach-Lenzmann \cite{BHL} derived a blow-up result with radial data for
\eqref{NFSE} in both $L^2$-supercritical and $L^2$-critical cases respectively, and Guo-Zhu \cite{GqZ} further completed the blow-up result with radial data for general dimensions and nonlinearities.
For cubic nonlinearity $|u|^{2}u$, Guo-Han-Xin \cite{GHX}
showed that the period boundary value problem of \eqref{NFSE} is globally well-posed. Ionescu-Pusateri \cite{IP} proved that the global small, smooth solutions of \eqref{NFSE} exists if $\alpha=\frac{1}{2}$. Guo-Huo \cite{GH} established the global well-posedness of \eqref{NFSE} if $1<\alpha<2$, where the key tri-linear estimates in Bourgain space played a important role.  Concerning to blow-up of solutions for fractional Schr\"odinger equations on $\mathbb{R}^n$, there are some interesting results. For $\alpha=1$, Fujiwara in \cite{Fujiwara} proved that there is no global weak solution; Fino-Dannawi-Kirane in \cite{FDK1,FDK2} considered the blow up of mild solutions for $0<\alpha<2$. Motivated by these work, we are interested in blow up of strong solutions of general fractional Schr\"odinger equations on Riemannian manifolds.

In this paper, we get the blow-up results for a class of Schr\"odinger equations on Riemannian manifolds with nonnegative Ricci curvature. To be precise, let $(M, g)$ be a complete manifold of dimension $n$ with nonnegative Ricci curvature. Denote by $d$ the geodesic distance and $\mu$ the Riemannian measure. By the Bishop-Gromov inequality, $M$ satisfies the doubling condition: there exists constant $C>0$ such that
$$V(x, 2r) \le C V(x,r), \hspace{1cm} \forall r>0, x\in M, $$
where $V(x,r)$ is the volume of the geodesic ball centered at $x$ with radius $r$. In this paper, we will use the Einstein's summation convention. Then in local coordinates, the Laplace-Beltrami operator can be expressed as
\begin{equation}\label{laplace-beltami}
  \Delta =  \frac{1}{\sqrt{ \det g}} \frac{\partial}{\partial x_j} \left( \sqrt{\det g} g^{jk} \frac{\partial}{\partial x_k} \right)
\end{equation}
where $(g^{jk})_{1\le j, k\le n}$ is the inverse matrix of $(g_{jk})_{1\le j, k \le n}$. For any $p\in M$, consider the normal coordinates in which the Riemannian metric can be written as
$$ g=dr^2+ r^2 g_{\alpha \beta}(r, \theta) d\theta^\alpha d\theta^\beta. $$
Thus for any distance function $f(r)=f(d(x,p))$, we have by\eqref{laplace-beltami}
\begin{equation}\label{radial function}
  \Delta f(r)= f''(r)+ \frac{n-1}{r} f'(r) + \frac{1}{\sqrt{G}} \frac{\partial \sqrt{G}}{\partial r} f'(r)
\end{equation}
where $G=\det(g_{\alpha,\beta})$.
Denote by $p_t(x,y)$ the Schwartz kernel of the heat semigroup $e^{t\Delta}$. According to \cite{LY}, the heat kernel $p_t(x,y)$ satisfies the Gaussian upper bound:
\begin{equation}\label{heat kernel}
  p_t(x,y) \le \frac{C}{V(x, \sqrt{t})} \exp \left( -c\frac{d^2(x,y)}{t} \right), \hspace{1cm} \forall t>0, x,y \in M.
\end{equation}
Then, by \cite{Grigoryan}, the following estimates also hold by \eqref{heat kernel} and doubling properties
\begin{equation}\label{heat kernel 2}
 |\Delta_y p_t(x,y)| \le \frac{C}{V(x, \sqrt{t})} \frac{1}{t} \exp \left( -c\frac{d^2(x,y)}{t} \right), \hspace{1cm} \forall t>0, x,y \in M.
\end{equation}
Moreover, by \cite{LY}(page 163 Theorem 1.3(i)), the following gradient estimates also hold under our assumption,
\begin{equation}\label{Assumpation3}
 |\nabla_y p_t(x,y) | \le \frac{C}{V(x, \sqrt{t})} \frac{1}{\sqrt{t}} \exp \left( -c\frac{d^2(x,y)}{t} \right), \hspace{1cm} \forall t>0, x,y \in M.
\end{equation}

The main results of this paper is as follows:
\begin{theorem}
Let $0<\alpha< 2, \ n\ge 2$ and $s=[\frac{n}{2}]+1$. Consider \eqref{NFSE} on smooth complete $n-$dimensional manifolds $(M,g)$ with nonnegative Ricci curvature
\begin{eqnarray*}
		\left
		\{
		\begin{array}{l}
			 \mathrm{i} \partial_t u- (-\Delta)^{\frac{\alpha}{2}} u=F(u,\overline{u}),
			\\
			u(0,x)=\varphi_0(x)+\mathrm{i}\varphi_1(x),
		\end{array}
		\right.
	\end{eqnarray*}
where $\ \varphi_0(x),\ \varphi_1(x) \in$
$H^s(M)$. Assume that the following conditions holds
\begin{equation}\label{Assumpation1}
 (1) \left| \frac{r}{\sqrt{G}} \frac{\partial \sqrt{G}}{\partial r} \right| \le C \hspace{1cm} \forall r>0.
\end{equation}
\begin{equation}\label{Assumpation2}
  (2) V(x,r)\sim r^n \hspace{1cm} \forall x\in M ,~ r>0.
\end{equation}
If $\mathbf{Im} F=|u|^p, 1<p<1+ \frac{\alpha}{n}$ and $ \int_{M}\varphi_0(x)d\mu >0,$
then the local solution of \eqref{NFSE} will blow up in a finite time no matter how small the initial data is.
\end{theorem}
The proof of Theorem 1.1 crucially relies on constructing a new weight function and ODE inequalities. Different from Euclidian spaces, there is non-implicit formula for fractional Laplacian operator or heat kernel on general Riemannian manifolds. Hence, this work is not a trival extension from Euclidian spaces to manifolds with nonnegative Ricci curvature. By using Li-Yau's result \cite{LY}, namely, some upper bounds of $|\nabla^k p_t|_{C^0}, k=0,1,2$ ($p_t$ is the heat kernel), we could prove that local solutions with initial data in spaces $H^{[\frac{n}{2}]+1}(M)$ will blow up, where some ODE type estimates plays crucial role.
\begin{remark}\label{r}
For all $p>1$, it's easy to see that the local solution of  \eqref{NFSE} exists in $H^{s}(M)$ ($s=[\frac{n}{2}]+1$) by bootstrap argument and contraction mapping principle. Applying $\nabla^j$ ($j$ is a multi integer index) to (1.1) yields
\begin{equation*}
\partial_t \nabla^j u- (-\Delta)^{\frac{\alpha}{2}}\nabla^j u= \nabla^j F(u,\bar{u}).
\end{equation*}
Taking inner product with $\nabla^j{\bar{u}}$ on (1.1) and taking the imaginary part, and then summing it for $|j|\leq s$, we have
\begin{equation*}
\begin{split}
||u(t,\cdot)||^2_{H^s} &\lesssim ||u_0||^2_{H^s}+ \int^t_0 ||u(\tau)||^{2}_{H^s} ||u||^{p-1}_{L^\infty}d\tau
\\
& \lesssim ||u_0||^2_{H^s}+ \int^t_0 ||u(\tau)||^{p+1}_{H^s}d\tau,
\end{split}
\end{equation*}
where in the last inequality we use that $H^s \hookrightarrow L^\infty$ for $s=[\frac{n}{2}]+1>\frac{n}{2}$. For sufficiently small $T>0$, we could obtain $u(t,x) \in C([0,T];H^s(M))$. The uniqueness of solution follows from the similar energy estimates.
\end{remark}
The paper  is organized as follows: In Section 2, we prove some lemmas which play an important role in our proof; Our main results will be shown in Section 3; In Section 4, we give an example which satisfies our assumptions on the Riemannian metric.   Now we introduce some notations. If $f, g$ are two functions, we say $f\lesssim g$ if and only if there exists a constant $c>0$ such that $f\le c g$. We say $f\sim g$ if and only if there exits a constant $C>0$ such that $C^{-1}g\le f \le C g $. Set $f\wedge g= \min\{ f, g \} $.  The constant $C,c$ may change from line to line.

\section{\textbf{Preliminary Results}}
In this section, we will prove several lemmas which will be frequently used.
\begin{lemma}
Let $y>0$ and we have
$$ 1\wedge y \sim \frac{1}{1+ y^{-1} }. $$
\end{lemma}
\begin{proof}
It is direct to check that
$$ \frac{1\wedge y}{2}\le \frac{1}{1+ y^{-1} } \le  1\wedge y. $$
\end{proof}

\begin{lemma}
Let $M$ be a $n-$dimensional complete Riemannian manifold with nonnegative Ricci curvature for $n\ge 2$.   Set $\alpha, \gamma>0$ and  the following hold :\\
(1)When $\gamma>\frac{n}{2}$, we have
$$ \int_M [t^{\frac{2}{\alpha}}+d^2(x,y)]^{-\gamma} d\mu(y) \lesssim t^{\frac{n-2\gamma}{\alpha}}, \hspace{1cm} \forall \alpha, t, R >0, x\in M.  $$
(2)When $0\le \gamma< n$, we have
$$ \int_{B(x,R)} d(x,y)^{-\gamma} d\mu(y) \lesssim R^{n-\gamma}, \hspace{1cm} \forall \alpha, t, R >0, x\in M.   $$
(3)When $\gamma>n$, we have
$$ \int_{M\backslash B(x,R)} d(x,y)^{-\gamma} d\mu(y) \lesssim R^{n-\gamma}, \hspace{1cm}  \forall \alpha, t, R >0, x\in M.   $$
\end{lemma}
\begin{proof}The proof relies on the ring decomposition and facts that the volume of geodesic balls grow polynomially. \\
(1) When $\gamma>\frac{n}{2}$, we have by the ring decomposition
\begin{align*}
  \int_M [t^{\frac{2}{\alpha}}+d^2(x,y)]^{-\gamma} d\mu(y) & =\sum_{j\in \mathbb{Z}} \int_{B(x, 2^j t^{\frac{1}{\alpha}})\backslash B(x, 2^{j-1} t^{\frac{1}{\alpha}}) } [t^{\frac{2}{\alpha}}+d^2(x,y)]^{-\gamma} d\mu(y)  \\
  & \lesssim  \sum_{j\in \mathbb{Z}} [ t^{\frac{2}{\alpha}}+ 2^{2j-2} t^{\frac{2}{\alpha}} ]^{-\gamma} 2^{jn} t^{\frac{n}{\alpha}} \\
   & \lesssim t^{\frac{n-2\gamma}{\alpha}} \sum_{j\in \mathbb{Z}} [ 1+2^{2j-2} ]^{-\gamma} 2^{jn}.
\end{align*}
The above series converges as long as $\gamma>\frac{n}{2}$.\\
(2) For $0\le \gamma<n$,  it is easy to check
\begin{align*}
  \int_{B(x,R)} d(x,y)^{-\gamma} d\mu(y) & =\sum_{j\ge 0} \int_{B(x, 2^{-j} R)\backslash B(x, 2^{-j-1} R) }  d(x,y)^{-\gamma} d\mu(y)  \\
   & \lesssim \sum_{j\ge 0}  2^{(j+1)\gamma} R^{-\gamma} 2^{-jn} R^{n} \lesssim R^{n-\gamma}.
\end{align*}
(3) Similarly we have for $\gamma>n$
\begin{align*}
  \int_{M\backslash B(x,R)} d(x,y)^{-\gamma} d\mu(y) & =\sum_{j\ge 0} \int_{B(x, 2^{j+1} R)\backslash B(x, 2^{j} R) }  d(x,y)^{-\gamma} d\mu(y)  \\
   & \lesssim \sum_{j\ge 0}  2^{-j\gamma} R^{-\gamma} 2^{(j+1)n} R^{n} \lesssim R^{n-\gamma}.
\end{align*}
\end{proof}
The following estimates about the weight functions play an important role in our proof. Note that when $M=\mathbb{R}^n$, the following estimates are essentially known in \cite{GO}, \cite{MYZ}. Now we give the proof on manifolds.
\begin{lemma}
Let $M$ be a $n-$dimensional complete Riemannian manifold with nonnegative Ricci curvature for $n\ge 2$ and  $h(t,x)=\frac{t^{1+\frac{n}{\alpha}}}{(d^2(x,p)+t^{\frac{2}{\alpha}})^{\frac{n+\alpha}{2}}}$ for $x, p \in M$. For $0<\alpha<2$, we have
\begin{equation*}
|(-\Delta)^{\frac{\alpha}{2}} h(t,x)| \lesssim t^{-1} h(t,x), \hspace{1cm} \forall x\in M, ~ t>0.
\end{equation*}
\end{lemma}
\begin{proof}
To start with, we have by \cite{Pazy}(p.72, Theorem 6.9)
\begin{align*}
  |(-\Delta)^{\frac{\alpha}{2}} h(t,x)| & =\frac{\sin \frac{\alpha}{2}\pi}{\pi} \int_0^\infty s^{\frac{\alpha}{2}-1} (-\Delta)( sI -\Delta )^{-1} h(t, x) ds   \\
   & = -\frac{\sin \frac{\alpha}{2}\pi}{\pi} \int_0^\infty s^{\frac{\alpha}{2}-1} ( sI -\Delta )^{-1} \Delta h(t, x) ds \\
   & = \int_M K(x, y) \Delta_y h(t, y) d\mu(y),
\end{align*}
where $K(x, y)$ is the integral kernel of $ -\frac{\sin \frac{\alpha}{2}\pi}{\pi} \int_0^\infty s^{\frac{\alpha}{2}-1} ( sI -\Delta )^{-1} ds$ \\
By Lemma 2.1 we have
$$ \frac{t^{1+\frac{n}{\alpha}}}{[d(x,p)^2+t^{\frac{2}{\alpha}}]^{\frac{n+\alpha}{2}}} =\frac{1}{[1+\left(\frac{d(x,p)}{t^{\frac{1}{\alpha}}}\right)^2]^{\frac{n+\alpha}{2}}} \sim \left[ 1\wedge \left(\frac{t^{\frac{1}{\alpha}}}{d(x,p)} \right)^2 \right]^{\frac{n+\alpha}{2}}.  $$
Thus it is sufficient to prove
\begin{equation}\label{integral form lemma}
  \left| \int_M K(x, y) \Delta_y h(t, y)   d\mu(y) \right| \lesssim t^{-1} \left[ 1\wedge \left(\frac{t^{\frac{1}{\alpha}}}{d(x,p)} \right)^2 \right]^{\frac{n+\alpha}{2}} .
\end{equation}
Indeed, we have to distinguish two cases.\\
\textbf{Case I: }$d(x,p)\le t^{\frac{1}{\alpha}}.$ \\
Note that in local coordinate, we have by \eqref{radial function}
$$ \Delta h(t,y)= \frac{\partial^2 h(t,r)}{\partial r^2} + \frac{n-1}{r} \frac{\partial h(t,r)}{\partial r} + \frac{1}{\sqrt{G}} \frac{\partial \sqrt{G}}{\partial r} \frac{\partial h(t,r)}{\partial r} $$
where $r=d(y,p)$. According to the definition of $h$, we obtain
\begin{equation*}
  \begin{split}
  &\frac{\partial h(t, r)}{ \partial r } = - \frac{(n+\alpha) t^{1+\frac{n}{\alpha}} r }{ [t^{\frac{2}{\alpha}} +r^2 ]^{\frac{n+\alpha}{2}+1} },  \\
  & \frac{\partial^2 h(t, r)}{ \partial r^2 } =  - \frac{(n+\alpha) t^{1+\frac{n}{\alpha}} }{ [t^{\frac{2}{\alpha}} +r^2 ]^{\frac{n+\alpha}{2}+1} }  +  \frac{(n+\alpha)(n+\alpha+2) t^{1+\frac{n}{\alpha}} r }{ [t^{\frac{2}{\alpha}} +r^2 ]^{\frac{n+\alpha}{2}+2} }.
  \end{split}
\end{equation*}
In turn, we have by \eqref{Assumpation1}
$$ |\Delta h(t,y)| \lesssim \frac{t^{1+\frac{n}{\alpha}} }{[ t^{\frac{2}{\alpha}} +r^2 ]^{\frac{n+\alpha}{2}+1} }. $$
Moreover, by \cite{Pazy}(p.8) and \eqref{heat kernel}, we obtain
\begin{align*}
  \left|\int_M K(x, y) \psi(y) d\mu(y)\right|= & c\left|  \int_0^\infty  s^{\frac{\alpha}{2} -1} (sI- \Delta)^{-1}  \psi  ds \right| \\
          = & c \left| \int_0^\infty s^{\frac{\alpha}{2} -1}  \int_0^\infty  e^{-s\tau} e^{\tau \Delta} \psi d\tau ds    \right| \\
          = & c' \left| \int_0^\infty e^{\tau \Delta} \psi \tau^{-\frac{\alpha}{2}} d\tau    \right| \\
          \lesssim & \int_0^\infty \int_M \frac{\tau^{-\frac{\alpha}{2}}}{V(x, \sqrt{\tau})}  \exp \left( -c\frac{d^2(x,y)}{\tau}  \right) |\psi(y) | d\mu(y) d\tau \\
           \lesssim & \int_M d(x, y)^{2-n-\alpha} |\psi(y)| d\mu(y).
\end{align*}
As a result, it follows
$$  \left| \int_M K(x, y) \Delta_y h(t, y)   d\mu(y) \right| \lesssim  \int_M \frac{t^{1+\frac{n}{\alpha}}}{[ t^{\frac{2}{\alpha}} +d(y,p)^2 ]^{\frac{n+\alpha}{2} +1 }}   d(x,y)^{2-n-\alpha }  d\mu(y) \triangleq I .  $$
Set
$$ I= \int_{d(x,y)\le t^{\frac{1}{\alpha}} } + \int_{d(x,y)> t^{\frac{1}{\alpha}} } \triangleq I_1 + I_2. $$
For $I_1$, we have by Lemma 2.2
$$ I_1\lesssim t^{1+\frac{n}{\alpha}-\frac{2}{\alpha}(\frac{n+\alpha}{2}+1)} \int_{d(x,y)\le t^{\frac{1}{\alpha}}}  d(x,y)^{2-n-\alpha }  d\mu(y) \lesssim t^{-1}.  $$
Note that we have used the assumption $0<\alpha<2$.\\
For $I_2$, we obtain
\begin{align*}
  I_2 & \lesssim t^{1+\frac{n}{\alpha} + \frac{1}{\alpha}(2-n-\alpha)} \int_{d(x,y)> t^{\frac{1}{\alpha}} } [ t^{\frac{2}{\alpha}} +d(y,p)^2 ]^{-\frac{n+\alpha}{2}-1} d\mu(y)  \\
   & \lesssim t^{\frac{2}{\alpha}} \int_M [ t^{\frac{2}{\alpha}} +d(y,p)^2 ]^{-\frac{n+\alpha}{2}-1} d\mu(y) \lesssim t^{-1}.
\end{align*}
We have used Lemma 2.2 in the last step. Thus \eqref{integral form lemma} holds for $d(x, p)\le t^{\frac{1}{\alpha}}.$ \\
\textbf{Case II: }$d(x,p)> t^{\frac{1}{\alpha}}.$\\
Let $\varphi_1(t), \varphi_2(t), \varphi_3(t)$ be the partition of unity, i.e.,
$$  \varphi_1(t)+ \varphi_2(t) + \varphi_3(t) =1, \hspace{0.3cm} \text{and} \hspace{0.3cm} 0\le  \varphi_k(t)\le 1 \hspace{0.5cm} \forall t>0, k=1, 2, 3. $$
$\varphi_1$ is supported in $0<t < \frac{3}{4} d(x, p)$, $\varphi_2$ is supported in $\frac{1}{2}d(x, p) <t < 2 d(x, p)$ and equals 1 for $\frac{3}{4} d(x, p) \le t \le \frac{5}{4} d(x, p)$, $\varphi_1$ is supported in $\frac{5}{4} d(x, p)<t.$ Moreover, they are smooth functions. For more properties of partition of unity we refer the readers to \cite{Hormander}(p.25).\\
Set for $k=1, 2, 3$
$$ \Pi_k= \int_M K(x, y) \varphi_k(d(x, y)) \Delta_y h(t, y)  d\mu(y). $$
For $\Pi_1$, by Lemma 2.2 we obtain
\begin{align*}
 |\Pi_1| &  \lesssim  \int_M d(x,y)^{2-n-\alpha } |  \Delta_y h(t, y) | d\mu(y). \\
& \lesssim t^{1+\frac{n}{\alpha}} [ t^{\frac{2}{\alpha}} +d(x,p)^2 ]^{-\frac{n+\alpha}{2}-1} \int_{d(x,y)\le \frac{3d(x,p)}{4} }  d(x,y)^{2-n-\alpha }  d\mu(y) \\
   & \lesssim t^{1+\frac{n}{\alpha}} d(x ,p)^{-n-2\alpha} \lesssim \frac{t^{\frac{n}{\alpha}}}{d(x ,p)^{n+\alpha}}.
\end{align*}
We have used the fact $d(y,p)\ge d(x,p)-d(x,y) \ge \frac{1}{4} d(x, p) $ in the first inequality. \\
For $\Pi_3$, we have
\begin{align*}
  |\Pi_3| &  \lesssim t^{1+\frac{n}{\alpha}} \int_{d(x,y)> \frac{5d(x,p)}{4} } [ t^{\frac{2}{\alpha}} +d(y,p)^2 ]^{-\frac{n+\alpha}{2}-1}  d(x,y)^{2-n-\alpha }  d\mu(y) \\
  &\lesssim t^{1+\frac{n}{\alpha}} d(x,p)^{-n-\alpha} \int_{d(x,y)> \frac{5d(x,p)}{4} } [ t^{\frac{2}{\alpha}} +d(y,p)^2 ]^{-\frac{n+\alpha}{2}-1}  d(x,y)^{2 }  d\mu(y) \\
   & \lesssim t^{1+\frac{n}{\alpha}} d(x,p)^{-n-\alpha} (J_1+J_2),
\end{align*}
where
$$ J_1=\int_{d(y,p)\le 2d(x,p) \bigcap d(x,y)> \frac{5d(x,p)}{4} } [ t^{\frac{2}{\alpha}} +d(y,p)^2 ]^{-\frac{n+\alpha}{2}-1}  d(x,y)^{2 }  d\mu(y),  $$
$$ J_2= \int_{d(y,p)> 2d(x,p) \bigcap d(x,y)> \frac{5d(x,p)}{4} } [ t^{\frac{2}{\alpha}} +d(y,p)^2 ]^{-\frac{n+\alpha}{2}-1}  d(x,y)^{2 }  d\mu(y). $$
Note that for $J_1$ the following holds
\begin{align*}
  J_1 & \lesssim d(x,p)^2 \int_{B(p, 2d(x,p)) } d(x, p)^{-n-\alpha-2 } d\mu(y) \\
      & \lesssim d(x,p)^{-\alpha},
\end{align*}
where we have used the facts $d(x, y)\le d(x, p)+ d(y, p)\le 3 d(x, p)$ and $d(y, p) \ge d(x, y) - d(x, p) \ge \frac{1}{4} d(x, p)$.  \\
Moreover
\begin{align*}
  J_2 & \lesssim  \int_{M\backslash B(p, 2d(x,p))}  d(y, p)^{-n-\alpha-2 } d(y,p)^2 d\mu(y) \\
      & \lesssim d(x,p)^{-\alpha}.
\end{align*}
We have used the facts $d(x, y)\le d(x, p)+ d(y, p)\le \frac{3}{2} d(y, p). $ As a result, we obtain
$$ |\Pi_3|\lesssim \frac{t^{\frac{n}{\alpha}}}{d(x ,p)^{n+\alpha}}. $$
Now we are ready to deal with $\Pi_2$. To this end we need the integration by parts.
\begin{align*}
  \Pi_2 = &  \int_M K(x, y) \varphi_2(d(x, y)) \Delta_y h(t, y) d\mu(y) \\
        = &  \int_M \Delta_y [ K(x, y) \varphi_2(d(x, y)) ] h(t, y) d\mu(y) \\
        = & \int_M [ \Delta_yK(x, y) \varphi_2(d(x, y)) + 2\langle \nabla_y K(x, y), \nabla_y \varphi_2(d(x, y))\rangle_g \\
          & +  K(x, y) \Delta_y \varphi_2(d(x, y)) ] h(t, y) d\mu(y)
\end{align*}
where $\langle \cdot \rangle_g$ is the Riemannian metric on the tangent spaces. \\
As above, we have by  \eqref{heat kernel 2}
\begin{align*}
 & \left| \int_M \Delta_yK(x, y) h(t,y)  d\mu(y) \right| \\
  = &  c \left| \int_0^\infty \int_M \tau^{-\frac{\alpha}{2}} \Delta_y p_\tau(x, y) h(t,y)  d\mu(y) d\tau  \right| \\
        \lesssim  &  \int_M \int_0^\infty  \frac{\tau^{-\frac{\alpha}{2}-1}}{V(x, \sqrt{\tau})}  \exp \left( -c\frac{d^2(x,y)}{\tau}  \right) d\tau |h(t,y)| d\mu(y)\\
      \lesssim    &  \int_M d(x, y)^{-n-\alpha}  |h(t,y)| d\mu(y).
\end{align*}
It follows that
\begin{align*}
   & \left| \int_M [ \Delta_yK(x, y) \varphi_2(d(x, y))] h(t, y) d\mu(y) \right| \\
  \lesssim &   t^{1+\frac{n}{\alpha}} \int_{\frac{d(x,p)}{2}\le d(x,y)\le 2d(x, p) } [ t^{\frac{2}{\alpha}} +d(y,p)^2 ]^{-\frac{n+\alpha}{2}}  d(x,y)^{-n-\alpha }  d\mu(y)  \\
     \lesssim    &   t^{1+\frac{n}{\alpha}} d(x,p)^{-n-\alpha }  \int_{\frac{d(x,p)}{2}\le d(x,y)\le 2d(x, p) } [ t^{\frac{2}{\alpha}} +d(y,p)^2 ]^{-\frac{n+\alpha}{2}}  d\mu(y)\\
     \lesssim & t^{1+\frac{n}{\alpha}} d(x,p)^{-n-\alpha } \left(\int_{E_1} [ t^{\frac{2}{\alpha}} +d(y,p)^2 ]^{-\frac{n+\alpha}{2}}  d\mu(y)  +\int_{E_2} [ t^{\frac{2}{\alpha}} +d(y,p)^2 ]^{-\frac{n+\alpha}{2}}  d\mu(y)\right)
\end{align*}
where $E_1=B(p, \frac{1}{2}d(x, p))$ and $E_2= \{y| \frac{d(x,p)}{2}\le d(x,y)\le 2d(x, p)\} \backslash B(p, \frac{1}{2}d(x, p)). $
Note that
\begin{align*}
   & \int_{E_1} [ t^{\frac{2}{\alpha}} +d(y,p)^2 ]^{-\frac{n+\alpha}{2}}  d\mu(y) \\
  \lesssim  & \int_{d(y, p)\le \frac{t^{\frac{1}{\alpha}}}{2} } [ t^{\frac{2}{\alpha}} +d(y,p)^2 ]^{-\frac{n+\alpha}{2}}  d\mu(y) + \int_{\frac{t^{\frac{1}{\alpha}}}{2}\le  d(y, p)\le \frac{d(x, p)}{2} } [ t^{\frac{2}{\alpha}} +d(y,p)^2 ]^{-\frac{n+\alpha}{2}}  d\mu(y) \\
  \lesssim & t^{-1} + \int_{\frac{t^{\frac{1}{\alpha}}}{2}\le  d(y, p)} [ t^{\frac{2}{\alpha}} +d(y,p)^2 ]^{-\frac{n+\alpha}{2}}  d\mu(y) \lesssim t^{-1}.
\end{align*}
On the other hand, since $d(y, p)\ge \frac{1}{2} d(x, p)$ for $y\in E_2$, we have
$$ \int_{E_1} [ t^{\frac{2}{\alpha}} +d(y,p)^2 ]^{-\frac{n+\alpha}{2}}  d\mu(y) \lesssim d(x, p)^{-n-\alpha} V(x, 2d(x, p)) \lesssim d(x, p)^{-\alpha}.  $$
As a result, we conclude
$$ \left| \int_M [ \Delta_yK(x, y) \varphi_2(d(x, y))] h(t, y) d\mu(y) \right| \lesssim \frac{t^{\frac{n}{\alpha}}}{d(x, p)^{n+\alpha}}.  $$
To proceed we need the following estimates:
\begin{align*}
 & \left| \int_M \langle \nabla_yK(x, y),  \nabla_y \varphi_2(d(x, y)) \rangle_g h(t,y)  d\mu(y) \right| \\
  = &  c \left| \int_0^\infty \int_M \tau^{-\frac{\alpha}{2}} \langle \nabla_y p_\tau (x, y),  \nabla_y \varphi_2(d(x, y)) \rangle_g  h(t,y)  d\mu(y) d\tau  \right| \\
        \lesssim  &  \int_{E_3} \int_0^\infty  \frac{\tau^{-\frac{\alpha}{2}-\frac{1}{2}}}{V(x, \sqrt{\tau})}  \exp \left( -c\frac{d^2(x,y)}{\tau}  \right) d\tau d(x, p)^{-1}  |h(t,y)| d\mu(y)\\
      \lesssim    & d(x, p)^{-1}  \int_{E_3} d(x, y)^{1-n-\alpha}  |h(t,y)| d\mu(y)
\end{align*}
where $E_3= \frac{d(x,p)}{2}\le d(x,y)\le \frac{3d(x, p)}{4} \bigcup \frac{5d(x, p)}{4} \le d(x,y)\le 2 d(x, p) $.  We have used assumption \eqref{Assumpation3} and the facts $|\varphi_2'(d(x, y))|\le c d(x,p)^{-1}$. See for example \cite{Hormander} for the properties of cut-off functions.\\
Then we have
\begin{align*}
 & \left| \int_M \langle \nabla_yK(x, y),  \nabla_y \varphi_2(d(x, y)) \rangle_g h(t,y)  d\mu(y) \right| \\
      \lesssim    &  t^{1+\frac{n}{\alpha}} d(x, p)^{-1} \int_{E_3} [ t^{\frac{2}{\alpha}} +d(y,p)^2 ]^{-\frac{n+\alpha}{2}}  d(x, y)^{1-n-\alpha}   d\mu(y)
      \\
      \lesssim  & t^{1+\frac{n}{\alpha}} d(x, p)^{-(n+\alpha)} \int_{E_3} [ t^{\frac{2}{\alpha}} +d(y,p)^2 ]^{-\frac{n+\alpha}{2}}  d\mu(y)
      \\
      \lesssim &  t^{1+\frac{n}{\alpha}} d(x, p)^{-(n+\alpha)} \cdot t^{-1}
      \\
      \lesssim & \frac{t^{\frac{n}{\alpha}}}{d(x, p)^{n+\alpha}}.
\end{align*}
Similarly we have
\begin{align*}
 & \left| \int_M K(x, y) \Delta_y \varphi_2(d(x, y))  h(t,y)  d\mu(y) \right| \\
  = &  c \left| \int_0^\infty \int_M  \tau^{-\frac{\alpha}{2}} p_\tau (x, y)  \Delta_y \varphi_2(d(x, y))   h(t,y)  d\mu(y) d\tau  \right| \\
        \lesssim  &  \int_{E_3} \int_0^\infty  \frac{\tau^{-\frac{\alpha}{2}}}{V(x, \sqrt{\tau})}  \exp \left( -c\frac{d^2(x,y)}{\tau}  \right) d\tau d(x, p)^{-2}  |h(t,y)| d\mu(y)\\
      \lesssim    & d(x, p)^{-2}  \int_{E_3} d(x, y)^{2-n-\alpha}  |h(t,y)| d\mu(y).
\end{align*}
$E_3$ is as above and we have used the facts that  in local coordinates
$$ \Delta_y \varphi_2(r)= \varphi''_2(r) + \frac{n-1}{r} \varphi'_2(r) + \frac{1}{\sqrt{G}} \frac{\partial \sqrt{G}}{\partial r} \varphi'_2(r) $$
where $r=d(x, y)$. Hence we have $ |  \Delta_y \varphi_2(r) |\le c d(x, p)^{-2} $ for $y\in E_3$. Thus we obtain
\begin{align*}
 & \left| \int_M K(x, y) \Delta_y \varphi_2(d(x, y))  h(t,y)  d\mu(y) \right| \\
      \lesssim    &  t^{1+\frac{n}{\alpha}} d(x, p)^{-2} \int_{E_3} [ t^{\frac{2}{\alpha}} +d(y,p)^2 ]^{-\frac{n+\alpha}{2}}  d(x, y)^{2-n-\alpha}   d\mu(y)
      \\
      \lesssim & t^{1+\frac{n}{\alpha}} d(x, p)^{-(n+\alpha)}\int_{E_3} [ t^{\frac{2}{\alpha}} +d(y,p)^2 ]^{-\frac{n+\alpha}{2}}  d\mu(y)
      \\
      \lesssim & t^{1+\frac{n}{\alpha}} d(x, p)^{-(n+\alpha)} t^{-1}
      \\
      \lesssim & \frac{t^{\frac{n}{\alpha}}}{d(x, p)^{n+\alpha}}.
\end{align*}
As a result, we have proved
$$ |\Pi_2|\lesssim \frac{t^{\frac{n}{\alpha}}}{d(x, p)^{n+\alpha}}. $$
Combining these estimates, we have proved the desired results.
\end{proof}

\section{\textbf{Proof of Theorem 1.1}}
For each $\phi_0, \phi_1 \in H^{[\frac{n}{2}]+1}$, there is a positive number $T_0$ and a unique solution $u \in C([0,T_0]; H^{[\frac{n}{2}]+1})$ for Cauchy's problem \eqref{NFSE}.
Let
\begin{equation*}
	u(t,x)=w(t,x)+ \mathrm{i}v(t,x),
\end{equation*}
and take the imaginary part of the nonlinear fractional Schr\"odinger equation \eqref{NFSE}, we have
\begin{eqnarray}
	\left
	\{
	\begin{array}{l}
		w_t - (-\Delta)^{\frac{\alpha}{2}} v=(w^2+v^2)^{\frac{p}{2}}\\
		t=0:w=\varphi_0(x).
		
	\end{array}
	\right.
\end{eqnarray}
Now fix a point $p\in M$. For $0<\alpha<2$, let $h(t,x)=\frac{t^{1+\frac{n}{\alpha}}}{(d^2(x,p)+t^{\frac{2}{\alpha}})^{\frac{n+\alpha}{2}}}$ and consider the following function,
\begin{equation}\label{phi}
	\phi(t)=\int_M h(T, x) w(t,x)d\mu(x),
\end{equation}
where $T=t+N$ and $N$ will be determined later. To start with, we have
\begin{align*}
	\phi'(t) &= \int_M h(T, x) w_t d\mu(x) + \int_M \partial_t h(T, x)w d\mu(x) \\
	& = \int_M h(T,x)((-\Delta)^{\frac{\alpha}{2}}v+(w^2+v^2)^{\frac{p}{2}}) d\mu(x) + \int_M \partial_t h(T, x)w d\mu(x).
\end{align*}
Note that $(-\Delta)^{\frac{\alpha}{2}}$ is self-adjoint. It follows
$$ \phi'(t)= \int_M \partial_t h(T, x)w + (-\Delta)^{\frac{\alpha}{2}}h(T,x)v d\mu(x) + \int_M h(T,x)(w^2+v^2)^{\frac{p}{2}} d\mu(x). $$
Now consider
$$ I= \int_M \partial_t h(T, x)w + (-\Delta)^{\frac{\alpha}{2}}h(T,x)v d\mu(x). $$
Note that
\begin{align*}
	|\partial_t h(T,x)| & = \left| (1+\frac{n}{\alpha}) \frac{T^{\frac{n}{\alpha}}}{(d^2(x,p)+T^{\frac{2}{\alpha}})^{\frac{\alpha+n}{2}}} - (1+\frac{n}{\alpha})\frac{T^{\frac{n+2}{\alpha}}}{(d^2(x,p)+T^{\frac{2}{\alpha}})^{\frac{\alpha+n}{2}+1}} \right| \\
	& \lesssim  \frac{1}{T} h(T,x).
\end{align*}
By Lemma 2.3, we have
\begin{align*}
	|I| & \lesssim \int_M T^{-1} h(T,x)(|w|+|v|)d\mu(x) \\
	& \lesssim T^{-1} \left( \int_M h(T,x) (w^2+v^2)^{\frac{p}{2}} d\mu(x) \right)^{\frac{1}{p}} \left( \int_M h(T,x) d\mu(x) \right)^{\frac{1}{p'}}\\
	& \lesssim T^{\frac{n}{\alpha p'}-1} \left( \int_M h(T,x) (w^2+v^2)^{\frac{p}{2}} d\mu(x)  \right)^{\frac{1}{p}},
\end{align*}
where we have used the  H\"older inequality in the second step. As a result, we have
$$ |I| \le \frac{1}{2} \int_M h(T,x)(w^2+v^2)^{\frac{p}{2}}d\mu(x) + C T^{\frac{n}{\alpha}-p'}.  $$
In turn, we obtain
$$ \phi'(t) \ge \frac{1}{2} \int_M h(T,x)(w^2+v^2)^{\frac{p}{2}}d\mu(x) - C T^{\frac{n}{\alpha}-p'}. $$
On the other hand,
\begin{align*}
	|\phi(t)| & \le \left(\int_M h(T,x) |w|^p d\mu(x) \right)^{\frac{1}{p}}
	\left(\int_M h(T,x)d\mu(x) \right)^{\frac{1}{p'}} \\
	& \lesssim T^{\frac{n}{\alpha p'}} \left( \int_M h(T,x) (w^2+v^2)^{\frac{p}{2}} d\mu(x) \right)^{\frac{1}{p}}.
\end{align*}
Then we obtain
$$ \phi'(t)\geq C\frac{|\phi(t)|^p}{T^{\frac{n}{\alpha}(p-1)}}-CT^{\frac{n}{\alpha}-p'}.$$
Therefore we have
$$ \phi(t) \ge \int_M h(N,x) \varphi_0(x)d\mu(x) - CN^{\frac{n}{\alpha}-p'+1}
+C\int_0^{t}\frac{|\phi(\tau)|^p}{(\tau+N)^{\frac{n}{\alpha}(p-1)}}d\tau. $$
Note that $\frac{n}{\alpha}-p' +1 <0$ whenever $p<1+\frac{\alpha}{n}.$ By the dominated convergence theorem, the following holds,
$$ \lim_{N\rightarrow \infty} \int_M h(N,x) \varphi_0(x)d\mu(x) = \int_M \varphi_0(x)d\mu(x). $$
Thus for $N$ large enough, we conclude that
\begin{equation}\label{fi0}
	\phi(t) \ge \frac{1}{2} \int_M \varphi_0(x)d\mu(x) +C\int_0^{t}\frac{|\phi(\tau)|^p}{(\tau+N)^{\frac{n}{\alpha}(p-1)}}d\tau.
\end{equation}
Denote
\begin{equation}\label{fi}
	\varphi(t):= \frac{1}{2} \int_M \varphi_0(x)d\mu(x) +C\int_0^{t}\frac{|\phi(\tau)|^p}{(\tau+N)^{\frac{n}{\alpha}(p-1)}}d\tau.
\end{equation}
It follows
$$ \varphi'(t) \ge C \frac{\varphi^p(t)}{(t+N)^{\frac{n}{\alpha}(p-1)}}. $$
Finally we get
$$ \varphi(t) \ge C [\varphi^{1-p}(0)+N^{1-\frac{n}{\alpha}(p-1)} - (t+N)^{1-\frac{n}{\alpha}(p-1)} ]^{\frac{-1}{p-1}}. $$
Set $$ t_*= (N^{1-\frac{n}{\alpha}(p-1)}+\varphi(0)^{1-p})^{\frac{1}{1-\frac{n}{\alpha}(p-1)}}-N.$$
Hence, $$\varphi(t) \geq C \left( (t_*+N)^{1-\frac{n}{\alpha}(p-1)}-(t+N)^{1-\frac{n}{\alpha}(p-1)} \right)^{-\frac{1}{p-1}}.$$
Since $1-\frac{n}{\alpha}(p-1)>0$ when $1<p<1+\frac{\alpha}{n}$,
then $$\left( (t_*+N)^{1-\frac{n}{\alpha}(p-1)}-(t+N)^{1-\frac{n}{\alpha}(p-1)} \right)^{-\frac{1}{p-1}} \rightarrow +\infty, \ \mathrm{if} \ t\rightarrow t_*^{-}.$$
Therefore, $$\varphi(t)\rightarrow +\infty, \mathrm{if} \ t\rightarrow t_*^{-} .$$
Seeing from \eqref{fi0}-\eqref{fi}, then $\phi(t)$ tends to infinity if $t\rightarrow t_*$. A direct calculation tells us
\begin{equation*}
\begin{split}
||h(T,x)||_{L^2(M)} & \lesssim T^{1+\frac{n}{\alpha}} \left( \int_{M} ((d(x,p))^2+T^{\frac{2}{\alpha}})^{-(n+\alpha)}d\mu(x) \right)^{\frac{1}{2}}
\\
& \lesssim T^{-1}.
\end{split}
\end{equation*}
By H\"older's inequality and \eqref{phi}, we derive that
\begin{equation*}
	\phi(t) \leq ||h(T,\cdot)||_{L^2(M)} || w(t,\cdot)||_{L^2(M)}.
\end{equation*}
As a result, $$||w(t,\cdot)||_{L^2_x(M)} \geq T\phi(t).$$
Therefore, $||w(t,\cdot)||_{L^2_x(M)}$ tends to infinity if $t\rightarrow t_*$, and $T_0<t_*$. Hence we have proved the theorem.

\section{\textbf{Appendix}}
Note that when $M$ is a rotationally symmetric manifold with nonnegative Ricci curvature, \eqref{Assumpation1} holds automatically. In fact, the Riemannian metric can be expressed as
$$g= dr^2+ \phi(r)^2 \tilde{g}_{\alpha \beta} (\theta) d\theta^\alpha d\theta^\beta, $$
where $\phi(0)=0, \phi'(0)=1, \phi(r)>0$ for every $r>0$ and $\tilde{g}$ is the standard metric on sphere. Thus we have $\sqrt{G}=(\frac{\phi(r)}{r})^{n-1} \sqrt{\det \tilde{g}}$. According to the nonnegative of Ricci curvature, we have $\frac{\phi''(r)}{\phi(r)}\le 0$ (see $\S 3.2.3$ of \cite{Petersen}). Then the facts $\phi(r)>0$ and $\phi'(0)=1$ give that $\phi'(r)$ is a decreasing function with $0\le \phi'(r)\le 1$ for $r>0$. Thus we have
$$ \phi'(r)\le \frac{1}{r} \int_0^r \phi'(s) ds = \frac{\phi(r)}{r}. $$
It is direct to check that $ \frac{r}{\sqrt{G}} \frac{\partial \sqrt{G}}{\partial r} =(n-1)\frac{r\phi'(r)-\phi(r)}{\phi(r)}  $ and hence \eqref{Assumpation1} holds.

Note also that by our proof the range of $p$ in Theorem 1.1 is determined by the volume growth of geodesic balls on manifolds and  $\frac{\alpha}{n} \rightarrow 0$ as $n\rightarrow \infty$. Thus it explains why the blow-up results of Schr\"odinger equations with polynomial nonlinearities on hyperbolic spaces should not be expected. As is known, hyperbolic spaces have the Ricci curvature -1 and the volume of balls grow exponentially.
\section{Acknowledgement}
The authors would like to thank the anonymous referee for his/her careful reading and valuable suggestions. The authors would like to express the gratitude to Prof. Yi Zhou for his helpful discussions and advice. The second author would also thank Prof. Changxing Miao and Prof. Jiqiang  Zheng for valuable recommendations. The first author is supported by the State Scholarship Fund of China Scholarship Council ( No.201808430121); Hunan Provincial Key Laboratory of Intelligent Processing of Big Data on Transportation, Changsha University of Science and Technology, Changsha, 410114, China; and the Swedish Research Council under grant no. 2016-06596 while the second author was in residence
at Institut Mittag-Leffler in Djursholm, Sweden during the year of 2019. She also thanks to Prof. Lars Andersson, for his hospital invitation and much help at Institut Mittag-Leffler. The second author is supported by the National Natural Science Foundation of China under Grant No.11901407.

\bibliographystyle{amsplain}

\end{document}